\documentclass{aims}
\usepackage{amsmath}
  \usepackage{paralist}
  \usepackage{graphics} 
  \usepackage{epsfig} 
 \usepackage[colorlinks=true]{hyperref}
\hypersetup{urlcolor=blue, citecolor=red}
\def\currentvolume{X}
 \def\currentissue{X}
  \def\currentyear{201X}
   \def\currentmonth{XX}
    \def\ppages{X--XX}

\newtheorem{theorem}{Theorem}[section]

\theoremstyle{definition}
\newtheorem{definition}[theorem]{Definition}
\newtheorem{remark}{Remark}
\newtheorem{example}{Example} 

\title[DNNP METHOD FOR SOLVING MOQP]
      {DOUBLY NONNEGATIVE RELAXATION METHOD FOR SOLVING MULTIPLE OBJECTIVE QUADRATIC PROGRAMMING PROBLEMS}

\author[ yanqin bai and chuanhao guo]{}


 \email{yqbai@shu.edu.cn}
 \email{guo-ch@shu.edu.cn}

\thanks{The research is supported by National Natural Science Foundation of
China (No. 11071158) and the Key Disciplines of Shanghai
Municipality (No. S30104).}

\begin{document}
\maketitle

\centerline{\scshape yanqin bai }
\medskip
{\footnotesize
 \centerline{Department of Mathematics, Shanghai
University}
   \centerline{Shanghai 200444, China}
} 

\medskip

\centerline{\scshape chuanhao guo}
\medskip
{\footnotesize
 \centerline{ Department of Mathematics, Shanghai
University}
   \centerline{Shanghai 200444, China}
}

\bigskip

 \centerline{(Communicated by the associate editor name)}

\begin{abstract}
Multicriterion optimization and Pareto optimality are fundamental
tools in economics. In this paper we propose a new relaxation method
for solving multiple objective quadratic programming problems.
Exploiting the technique of the linear weighted sum method, we
reformulate the original multiple objective quadratic programming
problems into a single objective one. Since such single objective
quadratic programming problem is still nonconvex and NP-hard in
general. By using the techniques of lifting and doubly nonnegative
relaxation, respectively, this single objective quadratic
programming problem is transformed to a computable convex doubly
nonnegative programming problem. The optimal solutions of this
computable convex problem are (weakly) Pareto optimal solutions of
the original problem under some mild conditions. Moreover, the
proposed method is tested with two examples and a practical
portfolio selection problem. The test problems are solved by
\texttt{CVX} package which is a solver for convex optimization. The
numerical results show that the proposed method is effective and
promising.
\end{abstract}
\section{Introduction}\label{sec: Introduction}
We consider multiple objective nonconvex quadratic programming
problems as follows
$$\rm{(MOQP)}\begin{array}{llll}
&\min &F(x)=(F_1(x),F_2(x),\ldots,F_p(x))^T&\\
&{\rm s.t.} &Ax=b,\\
&&x\geq 0,&\\
 \end{array}
$$
where $F_i(x)=x^TQ_ix+2c_i^Tx$ and $x\in R^n$ is the decision
variable. $Q_i\in R^{n\times n},\ c_i\in R^n,\ i=1,\ldots,p$,
$A=(a_1,a_2,\ldots,a_m)^T\in R^{m\times n}$ and
$b=(b_1,b_2,\ldots,b_m)\in R^m$ are given data. Without loss of
generality, $Q_i$ is symmetric and not positive semidefinite by
assumption.

Multi-objective programming (MOP) also known as multi-criteria
optimization, is the process of simultaneously optimizing two or
more conflicting objectives subject to certain constraints. (MOP)
problems are found in many fields, such as facility location and
optimal detector design \cite{bv2004}, image processing
\cite{dpbs2009}. Problem (MOQP) is a subclass of (MOP) problem and
arises in portfolio selection \cite{xl2002}, reservoir optimal
operation \cite{yy1991} and so on. Problem (MOQP) also can be viewed
as an extension of multiple objective quadratic-linear programming
(MOQLP) problem for which the objectives are a quadratic and several
linear functions and the constraints are linear functions which were
studied in \cite{ky1997,xl2002}.

Problem (MOP) does not have a single solution that simultaneously
minimizes each objective function. A tentative solution is called
Pareto optimal if it impossible to make one objective function
better off without necessarily making the others worse off. And
problem (MOP) may have many Pareto optimal solutions. For solving
problem (MOP), the linear weighted sum method is one of the most
widely used methods. The main idea is to choose the weighting
coefficients corresponding to objective functions. Then, problem
(MOP) can be transformed to a single objective one, and the Pareto
optimal solutions for problem (MOP) could be found by solving this
single objective problem with the appropriate weights. Ammar
\cite{a2008,a2009} investigates problem (MOQP) with fuzzy random
coefficient matrices. Under the assumption that the coefficient
matrices in objectives are positive semidefinite, some results are
discussed to deduce Pareto optimal solutions for fuzzy problem
(MOQP). However, many practical problems of this class of problems
are nonconvex in general. So, these two methods have certain
limitations in practical applications.

Burer \cite{b2009} proves that a large class of NP-hard nonconvex
quadratic program with a mix of binary and continuous variables can
be modeled as so called completely positive programs (CPP), i.e.,
the minimization of a linear function over the convex cone of
completely positive matrices subject to linear constraints (For more
details and developments of this technique, one may refer to
\cite{b2012,b2009,b2010,lfjwx2011}). In order to solve such convex
programs efficiently, a computable relaxed problem is obtained by
approximation the completely positive matrices with doubly
nonnegative matrices, resulting in a doubly nonnegative programming
\cite{b2010}, which can be efficiently solved by some popular
packages.

Motivated by the ideas of \cite{b2009,b2010}, we propose a new
relaxation method for solving problem (MOQP) by combining with the
linear weighted sum method. First of all, in virtue of the linear
weighted sum method, we first transform problem (MOQP) into a single
objective quadratic programming (SOQP) problem over a linearly
constrained subset of the cone of nonnegative orthant. Since problem
(SOQP) is a nonconvex in general, which is equivalently reformulated
as a completely positive programming problem, which is NP-hard.
Furthermore, a computable relaxed convex problem for this completely
positive programming problem is derived by using doubly nonnegative
relaxation technique, and resulting in a doubly nonnegative
programming (DNNP) problem. Based on the characteristics of optimal
solutions of problem (DNNP), a sufficient condition for (weakly)
Pareto optimal solutions for problem (MOQP) is proposed. Moreover,
the proposed method is tested with two examples and a practical
portfolio selection problem. The test problems are solved by
\texttt{CVX} package, which is a solver for convex optimization. The
numerical results show that the proposed method is effective and
promising.

The paper is organized as follows. In Section \ref{sec:
Preliminaries}, we recall some basic definitions and preliminaries
for (weakly) Pareto optimal solution and the linear weighted sum
method, respectively. In Section \ref{sec: Optimality conditions},
problem (MOQP) is transformed into problem (SOQP) by using the
linear weighted sum method. And some optimality conditions for
problem (MOQP) are established. In order to solve problem (SOQP)
effectively, problem (SOQP) is equivalently reformulated as a convex
problem (CP) which is further relaxed to a computable (DNNP) problem
in Section \ref{sec: Reformulation}. In Section \ref{sec: Numerical
experiments}, numerical results are given to show the performance of
the proposed method. Some conclusions and remarks are given in
Section \ref{sec: Concluding remarks}.
\subsection{Notation and terminology}\label{sec: Notation and terminology}
Let $\Omega=\{x\in R^n: Ax=b, x\geq 0\}$ be the feasible set of
problem (MOQP). Let $R^n_{+}$ (or $R^n_{++}$) denotes the cone of
nonnegative (or positive) vectors with dimension $n$, $S_n$ the cone
of all $n\times n$ symmetric matrices, $S_n^+$ the cone of all
$n\times n$ symmetric positive semidefinite matrices and $(S_n)^+$
the cone of all $n\times n$ symmetric matrices with nonnegative
elements. $C_n^*$ is the cone of all $n\times n$ completely positive
matrices, i.e.,
$$C_n^*=\{X\in R^{{n}\times{n}}:\ X=\sum\limits_{k\in
K}z^k(z^k)^T\}\cup\{0\},$$ where $\{z^k\}_{k\in K}\subset
R^n_+\backslash\{0\}$. For two vectors $x,\ y\in R^n$, $x\circ y$ is
a vector in $R^n$ with $x_iy_i$ is its $i$-th component. For a
matrix $M$, ${\rm{diag}(M)}$ is a column vector whose elements are
the diagonal elements of $M$. Given two conformal matrices $A$ and
$B$, $A\bullet B={\rm trace}(A^TB)$. For a given optimization
problem $(\ast)$, its optimal objective value is denoted by
$\rm{Opt(\ast)}$.
\section{Preliminaries}\label{sec: Preliminaries}
\subsection{Pareto optimal solutions}\label{sec: Pareto optimal solution}
In multi-objective optimization with conflicting objectives, there
is no unique optimal solution. A simple optimal solution may exist
here only when the objectives are non-conflicting. For conflicting
objectives one may at best obtain what is called Pareto optimal
solutions. For the sake of completeness, we restate the definitions
of some types of Pareto optimal solutions and ideal point from
\cite{h1994}.
\begin{definition}\label{def: Pareto optimal solution}
A solution $x^*$ is said to be Pareto optimal solution of problem
(MOQP) if and only if there does not exist another feasible solution
$x\in\Omega$ such that $F_i(x)\leq F_i(x^*)$ for all $i= 1,\ldots,p$
and $F_i(x)<F_i(x^*)$ for at least one index $i$.
\end{definition}
All Pareto optimal points lie on the boundary of the feasible region
$\Omega$. Often, algorithms provide solutions that may not be Pareto
optimal but may satisfy other criteria, making them significant for
practical applications. For instance, weakly Pareto optimal is
defined as follows.
\begin{definition}\label{def: weakly Pareto optimal solution}
A solution $x^*$ is said to be weakly Pareto optimal solution of
problem (MOQP) if and only if there does not exist another feasible
solution $x\in\Omega$ such that $F_i(x)<F_i(x^*)$ for all $i=
1,\ldots,p$.
\end{definition}
\begin{remark}\label{rem: relations for Pareto optimal solution and weakly Pareto optimal solution}
A solution is weakly Pareto optimal if there is no other point that
improves all of the objective functions simultaneously. In contrast,
a point is Pareto optimal if there is no other point that improves
at least one objective function without detriment to another
function. It is obvious that each Pareto optimal point is weakly
Pareto optimal, but weakly Pareto optimal point is not Pareto
optimal.
\end{remark}
In order to illustrate (weakly) Pareto optimal solution intuitively,
two examples are given as follows.
\begin{example}\label{example1 for Pareto optimal solution} Let $\Omega=[0, 3]$, and
$$F_1(x)=\frac{x^2}{2}+1,\ F_2(x)=\left\{\begin{array}{llll}-x+2, & x\in[0, 1],\\
1, & x\in(1, 2],\\
x-1, & x\in(2, 3].\end{array}\right.$$
\end{example}
\begin{figure}[htp]
\begin{center}
\includegraphics[width=3in]{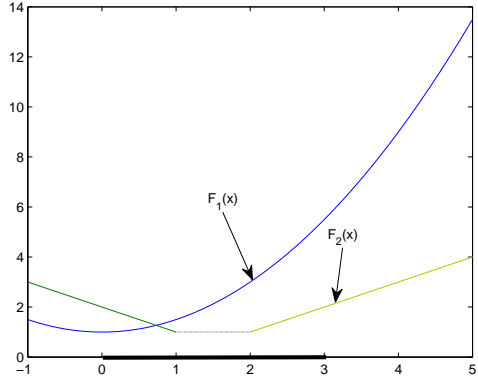}\\
\caption{Pareto optimal solution set [0, 1] and weakly Pareto
optimal solution set [0, 2]}\label{fig: example1 for Pareto optimal
solution}
\end{center}
\end{figure}
The design space for this problem is shown in Figure \ref{fig: example1 for Pareto optimal
solution}.
According to the above two definitions about (weakly) Pareto optimal
solution, we can easily get Pareto optimal solution set for this
problem is $[0, 1]$, and weakly Pareto optimal solution set is $[0,
2]$. Note that each Pareto optimal solution is weakly Pareto optimal
solution for this problem, since $[0, 1]\subseteq[0, 2]$, this also
shows the conclusion holds in Remark \ref{rem: relations for Pareto
optimal solution and weakly Pareto optimal solution}.

\begin{example}\label{example2 for Pareto optimal solution} Let $\Omega=[0, 4]$, and
$$F_1(x)=(x-1)^2+2,\ F_2(x)=(x-2)^2+1,\ F_3(x)=\|x\|.$$
\end{example}
\begin{figure}[htp]
\begin{center}
\includegraphics[width=3in]{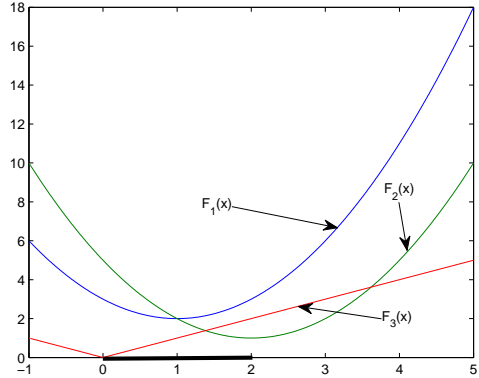}\\
\caption{Pareto optimal solution set [0, 2] and weakly Pareto
optimal solution set [0, 2]}\label{fig: example2 for Pareto optimal
solution}
\end{center}
\end{figure}

For this problem, the graphs for each objective function is plotted
in Figure \ref{fig: example2 for Pareto optimal solution}. By
Definitions \ref{def: Pareto optimal solution} and \ref{def: weakly
Pareto optimal solution}, it holds that Pareto optimal solution and
weakly Pareto optimal solution for this problem are the same, they
all equal to $[0, 2]$.

\subsection{The linear weighted sum method}\label{sec: The linear weighted sum method}
One useful way of getting the efficiency of problem (MOQP) is to
build a utility function \cite{xl2005} according to the decision
makers provided preference information, such that each solution
gained by this method is Pareto optimal solution of problem (MOQP).

The linear weighted function
\begin{equation}\label{eq: 1}
u(F(\cdot)):=\sum\limits_{i=1}^p\lambda_iF_i(\cdot)
\end{equation} is one of
the most widely used utility function, where weight $\lambda_i$
corresponding to objective functions $F_i(\cdot), \ i=1,\ldots,p$
satisfy the following conditions
\begin{equation}\label{eq: 2}
\sum\limits_{i=1}^p \lambda_i=1,\ \lambda_i\geq 0,\
i=1,\ldots,p,\end{equation} which is provided by the decision
makers, and weights $\lambda_i$ imply that the relative importance
for $F_i(\cdot)$ in the heart of the decision makers.

\section{Optimality conditions}\label{sec: Optimality conditions}
In this section, we derive a single objective quadratic programming
(SOQP) problem corresponding to problem (MOQP) by the linear
weighted sum method. And, some optimality conditions for problem
(MOQP) are proposed based on the optimal solutions for problem
(SOQP).

From (\ref{eq: 1}) and (\ref{eq: 2}), problem (MOQP) can be convert
to a single objective quadratic programming problem as follows
$$\rm{(SOQP)}\begin{array}{llll}
&\min &\sum\limits_{i=1}^p \lambda_i(x^TQ_ix+2c_i^Tx)&\\
&{\rm s.t.} &Ax=b,\\
&&x\geq 0.&\\
 \end{array}
$$

Similar to Corollary 3.1 in \cite{xl2005}, we can get the following
theorem, its proof can be found in \cite{xl2005}.
\begin{theorem}\label{theo: Pareto optimal solution}
Let $x^*$ be an optimal solution for problem (SOQP), it follows that
$x^*$ is a Pareto optimal solution (weakly Pareto optimal solution)
for problem (MOQP) if weight $\lambda>0$ ($\lambda\geq 0$,
$\lambda\neq 0$).
\end{theorem}

According to the above Theorem \ref{theo: Pareto optimal solution},
varying weight $\lambda$ consistently and continuously can result in
a subset of Pareto optimal (weakly Pareto optimal) set for problem
(MOQP). The following theorem presents a sufficient optimality
condition for weakly Pareto solution of problem (MOQP).

Firstly, we quote the following definition which will be used in the
sequel.
\begin{definition}\label{def: convex inclusion condition}
Let $\bar{x}\in\Omega$ and $Y=F(\Omega)$, if there exists a convex
set $H$ such that $$Y\subseteq H,\ {\rm{and}}\ H\cap
(F(\bar{x})-R^p_{++})=\emptyset,$$ then we say that $Y$ satisfies
the convex inclusion condition at $F(\bar{x})$.
\end{definition}
Based on the above convex inclusion condition, we have the following
theorem. The proof is omitted here for the reason that it is similar
to the one of Theorem 3.6 in \cite{xl2005}. 
\begin{theorem}\label{theo: weakly Pareto optimal solution}
Let $x^*$ be a weakly Pareto optimal solution for problem (MOQP),
and $Y$ satisfies the convex inclusion condition at $F(x^*)$. Then
there exists a weight $\bar{\lambda}\geq 0,\ \bar{\lambda}\neq 0$
such that $x^*$ is an optimal solution for problem (SOQP) with
$\bar{\lambda}$.
\end{theorem}
In particular, if problem (MOQP) is convex, Theorem \ref{theo:
weakly Pareto optimal solution} still holds without the convex
inclusion condition.

\section{Reformulation}\label{sec: Reformulation}
Note that problem (SOQP) is a nonconvex quadratic programming
problem in general, and thus it is NP-hard. If problem (SOQP) is
convex with appropriate coefficient $\lambda$, we may use some
popular convex packages to solve it directly. In the following
section, we will establish the computable convex reformulation for
problem (SOQP) when it is nonconvex, and the details are as follows.

\subsection{Completely positive reformulation}\label{sec: Completely positive representation}
Motivated by the ideas in \cite{b2009}, problem (SOQP) can be
reformulated as a completely positive programming problem. First,
the definition of completely positive \cite{bsm2003} is given as
follows.
\begin{definition}\label{def: completely positive}
A symmetric matrix $B$ of order $n$ is called completely positive if
one can find an integer $m$ and a matrix $V$ of size $n\times m$
with nonnegative entries such that $B=VV^T$, where the smallest
possible number $m$ is called the CP-rank of $B$.
\end{definition}
Based on above definition for completely positive, by using the
techniques in \cite{b2009}, problem (SOQP) can be reformulated as
the following completely positive programming problem
$$\rm{(CP)}\begin{array}{llll}
&\min &\sum\limits_{i=1}^p \lambda_i(Q_i\bullet X+2c_i^Tx)\\
&{\rm s.t.} &Ax=b,\\
&&{\rm{diag}}(AXA^T)=b\circ b,\\
&&\left(\begin{array}{lll}1 & x^T \\ x & X\end{array}\right)\in
C_{1+n}^*,
 \end{array}
$$ which is a convex programming problem. Similar to Theorem 2.6 in \cite{b2009}, the following theorem holds
immediately, for more details can be seen in \cite{b2009}.
\begin{theorem}\label{theo: problem (SOQP) and problem (CP)} $\it{Opt}(SOQP)=\it{Opt}(CP)$, and if
$(x^*,X^*)$ is an optimal solution for problem (CP), then $x^*$ is
in the convex hull of optimal solutions of problem (SOQP).
\end{theorem}
According to Theorem \ref{theo: problem (SOQP) and problem (CP)},
problem (SOQP) is equivalent to problem (CP). However, problem (CP)
is NP-hard, since there is a cone $C^*_{1+n}$ constraint, and check
whether or not a given matrix belong to $C^*_{1+n}$ is shown to be
NP-hard \cite{dg2011}, one must relax it in practice. Relaxing
problem (CP) in a natural way yields a doubly nonnegative
programming (DNNP) problem.
\subsection{Doubly nonnegative relaxation}\label{sec: Doubly nonnegative relaxation}
As mentioned above, in order to establish the doubly nonnegative
relaxation for problem (CP), the definition of doubly nonnegative is
given as follows.
\begin{definition}\label{def: doubly nonnegative}
If matrix $Z\in S_n$ is not only nonnegative but also positive
semidefinite, then $Z$ is called doubly nonnegative.
\end{definition}

Note that if $Z\in C^*_{n}$, it necessarily holds that $Z$ is doubly
nonnegative from the above Definitions \ref{def: completely
positive} and \ref{def: doubly nonnegative}. Moreover, the convex
cone $S_n^+$ is self-dual, and so is the convex cone $(S_n)^+$.
Hence, Diananda's decomposition theorem \cite{d1962} can be
reformulated as follows.
\begin{theorem}\label{theo: Diananda's decomposition theorem}For all $n$, we have $C_{n}^*\subseteq S_{n}^+\cap (S_{n})^+$.
The relationship $"\supseteq"$ for two sets
 holds if and only if $n\leq 4$.
\end{theorem}
Regardless of the dimension $n$, one always has the inclusion
$C_n^*\subseteq S_n^+\cap (S_n)^+$. Of course, in dimension $n\geq
5$ there are matrices which are doubly nonnegative but not
completely positive. The counterexample
$$M=\left[\begin{array}{lllllll}1& 1& 0& 0& 1\\
                                 1& 2& 1& 0& 0\\
                                 0& 1& 2& 1& 0\\
                                 0& 0& 1& 2& 1\\
                                 1& 0& 0& 1& 6\end{array}\right]$$
proposed by Diananda \cite{d1962} to illustrates this point.

Replacing $C^*_{1+n}$ by $S_{1+n}^+\cap (S_{1+n})^+$ according to
Theorem \ref{theo: Diananda's decomposition theorem}, problem (CP)
is relaxed to the following doubly nonnegative programming problem
$$\rm{(DNNP)}\begin{array}{llll}
&\min &\sum\limits_{i=1}^p \lambda_i(Q_i\bullet X+2c_i^Tx)\\
&{\rm s.t.} &Ax=b,\\
&&{\rm{diag}}(AXA^T)=b\circ b,\\
&&\left(\begin{array}{lll}1 & x^T \\ x & X\end{array}\right)\in
S_{1+n}^+\cap(S_{1+n})^+,
 \end{array}
$$
which is not only a convex problem but also can be solved in
polynomial time to any fixed precision from the theory of
interior-point methods.

Up to now, problem (MOQP) is reformulated as above problem (DNNP),
which can be solved by some popular package \texttt{CVX}. It is
obviously that problem (DNNP) is a relaxation form for problem
(MOQP).

In the last of this section, we will investigate the relationship
between optimal solutions for problems (MOQP) and (DNNP), i.e., a
sufficient condition for (weakly) Pareto optimal solutions of
problem (MOQP) based on the characteristics of optimal solutions for
problem (DNNP) is established in the following part.
\begin{theorem}\label{theo: optimal solution between problems}
Let $(x^*, X^*)$ be an optimal solution for problem (DNNP). If the
relationship $X^*=x^*(x^*)^T$ holds, then
${\it{Opt(DNNP)}}={\it{Opt(CP)}}.$ Moreover, $(x^*, X^*)$ is an
optimal solution for problem (CP).
\end{theorem}
\begin{proof}On one hand, from Theorem \ref{theo: Diananda's decomposition theorem}, it is obviously holds that
\begin{equation}\label{eq: 3}
{\rm{Opt}(CP)}\geq{\rm{Opt}(DNNP)}.
\end{equation}
On the other hand, from $X^*=x^*(x^*)^T$, and constraints of problem
(CP), we have $(x^*, X^*)$ is also a feasible solution for problem
(CP). Since problems (DNNP) and (CP) have the same objective
function, it follows that
\begin{equation}\label{eq: 4}
{\rm{Opt(CP)}}\leq\sum\limits_{i=1}^p\lambda_i(Q_i\bullet
X^*+2c_i^Tx^*)={\rm{Opt(DNNP)}}.
\end{equation}
Thus, combining (\ref{eq: 3}) and (\ref{eq: 4}), we have
$${\rm{Opt(CP)}}={\rm{Opt(DNNP)}}.$$
Again from problems (CP) and (DNNP) have the same objective
function, it holds that $(x^*, X^*)$ is an optimal solution for
problem (CP).
\end{proof}
\begin{remark}\label{rem: optimal solution between problems}
It holds that ${\it{Opt(SOQP)}}={\it{Opt(CP)}}$ from Theorem
\ref{theo: problem (SOQP) and problem (CP)}. Let $(x^*,X^*)$ be an
optimal solution for problem (DNNP), if $X^*=x^*(x^*)^T$, by Theorem
\ref{theo: optimal solution between problems}, we have
${\it{Opt(CP)}}={\it{Opt(DNNP)}}.$ Thus, we get
${\it{Opt(SOQP)}}={\it{Opt(DNNP)}}$ under the condition
$X^*=x^*(x^*)^T$. Furthermore, since problems (SOQP) and (DNNP) have
the same objective function, again from Theorem \ref{theo: problem
(SOQP) and problem (CP)}, we can conclude that $x^*$ is an optimal
solution for problem (SOQP). From Theorem \ref{theo: Pareto optimal
solution}, we further know that $x^*$ is a Pareto optimal solution
(or weakly Pareto optimal solution) for problem (MOQP) if
$\lambda>0$ (or $\lambda\geq 0$).
\end{remark}
\section{Numerical experiments}\label{sec: Numerical experiments}
In this section, in order to show the effectiveness of our proposed
method, some examples are tested and corresponding numerical results
are reported. To solve test problems, we use \texttt{CVX}
\cite{gb2011}, a package for specifying and solving convex programs.
The software is implemented using MATLAB R2011b on Windows 7
platform, and on a PC with Intel(R) Core(TM) i3-2310M CPU 2.10 GHz.

 In the following numerical experiments, three examples are solved by using
the proposed method, respectively. The first example is a given
two-dimension problem, which has four nonconvex objective functions.
The second example is a five-dimension nonconvex problem with five
objective functions. Note that its coefficients are generated by
MATLAB function \texttt{randn($\cdot$)}. The last example is a
practical portfolio selection problem, which is taken from
\cite{xl2002}. The weighted coefficient $\lambda$ is generated by
the following procedure

{\texttt{lambda=zeros(p,1)}};

{\texttt{while lambda(p)==0}}

{\texttt{lambda(1:p-1)=rand(p-1,1); s=sum(lambda)}};

{\texttt{if s<1 lambda(p)=1-s; end}}

{\texttt{end}}\\
 where $p$ is the number of objective functions.
\begin{example}\label{constructed problem}
First of all, a two-dimension problem with four objective functions
is tested. The corresponding coefficients $Q_i,\ c_i,\ i=1,\ldots,4$
and $A,\ b$ are given in Table \ref{tab: Coefficient data for
constructed problem}.
\begin{table}[htbp]
\centering \caption{Coefficients for Example \ref{constructed
problem}}\scriptsize
\begin{tabular}{ lllllllllllllllllllllllllll }
 \hline\hline
   $Q_1=\left[\begin{array}{ll}1 & 3 \\ 3 & 8\end{array}\right]$&
   $Q_2=\left[\begin{array}{lllll}2 & 0\\ 0 & -2\end{array}\right]$&
   $Q_3=\left[\begin{array}{lllll}0 & 2\\ 2 & 0\end{array}\right]$&
   $Q_4=\left[\begin{array}{llllll}-1 & 0\\ 0 & 0\end{array}\right]$\\
 \hline
 $c_1=[2.5\ \ -0.5]^T$ & $c_2=[-1\ \ -1.5]^T$ & $c_3=[1 \ \ 0.5]^T$ & $c_4=[1\ \ 1.5]^T$\\
 \hline
 $A=[1 \ \  1]$&$b=1$&&\\
 \hline\hline
\end{tabular}
\label{tab: Coefficient data for constructed problem}
\end{table}
\end{example}
First, by relative simple computation, we obtain optimal solutions
for each objective function which is minimized independently. The
corresponding optimal numerical results of each objective function
are given in Table \ref{tab: Optimal numerical results for each
objective function of constructed problem}.
\begin{table}[htbp] \centering \caption{Optimal numerical results
for each objective function of Example \ref{constructed
problem}}\scriptsize
\begin{tabular}{ lllllllllllllllllllllllllll }
 \hline\hline
       &$F_1(x^*)$ &&&&  &$F_2(x^*$) &&&& &$F_3(x^*$) &&&& &$F_4(x^*$)\\
 \hline
 FV  &$5.6667$          &&&&        &$-5$   &&&&     &$1$ &&&&   &$1$\\
 $x^*$ &$(0.6667\  0.3333)^T$  &&&& &$(0\ 1)^T$ &&&& &$(0\ 1)^T$ &&&& &$(1\ 0)^T$\\
 \hline\hline
\end{tabular}
\label{tab: Optimal numerical results for each objective function of
constructed problem}
\end{table}

In Table \ref{tab: Optimal numerical results for each objective
function of constructed problem}, the labels FV and $x^*$ denotes
optimal values and optimal solutions corresponding to each objective
function, respectively. The results in Table \ref{tab: Optimal
numerical results for each objective function of constructed
problem} show that the objective functions $F_2(x)$ and $F_3(x)$
have the same optimal solution, which is different from the other
two functions optimal solutions. These imply that we can not find a
single solution that simultaneously optimizes each objective
function. Moreover, note that these four optimal solutions
corresponding to each objective function are all weakly Pareto
optimal solutions for Example \ref{constructed problem}.
\begin{table}[htbp]
\centering \caption{Optimal numerical results for Example
\ref{constructed problem} in (DNNP)}\scriptsize
\begin{tabular}{ llllllllllllllll }
\hline\hline
 $\lambda$ &FV &$x^*$ &$X^*$\\
 \hline
 $(0.5472 \   0.1386  \  0.1493 \   0.1649)$
 &$\left[\begin{array}{ll}5.8345\\-2.8494\\2.4106\\2.3849\end{array}\right]$
 &$\left[\begin{array}{ll}0.4301\\   0.5699\end{array}\right]$
 &$\left[\begin{array}{ll}0.1850  &
 0.2451\\    0.2451  &  0.3248\end{array}\right]$\\
\hline
 $(0.3500\    0.1966\    0.2511\    0.2023)$
 &$\left[\begin{array}{ll}7.0000\\-5.0000\\1.0000\\3.0000\end{array}\right]$
 &$\left[\begin{array}{ll} 0.0000\\    1.0000\end{array}\right]$
 &$\left[\begin{array}{ll} 0.0000  &  0.0000\\
    0.0000 &   1.0000\end{array}\right]$\\
\hline
 $(0.0759\    0.0540\    0.5308\    0.3394)$
 &$\left[\begin{array}{ll}6.0000\\0.0000\\ 2.0000\\1.0000\end{array}\right]$
 &$\left[\begin{array}{ll} 1.0000\\    0.0000\end{array}\right]$
 &$\left[\begin{array}{ll} 1.0000  &  0.0000\\
    0.0000 &   0.0000\end{array}\right]$\\
\hline
 $(0.2417\    0.4039\    0.0965\    0.2579)$
 &$\left[\begin{array}{ll}5.6670\\-1.6105\\2.5513\\1.8625\end{array}\right]$
 &$\left[\begin{array}{ll} 0.6779\\    0.3221\end{array}\right]$
 &$\left[\begin{array}{ll} 0.4596  &  0.2183\\
    0.2183 &   0.1037\end{array}\right]$\\
\hline
 $(0.5752\    0.0598\    0.2348\    0.1302)$
 &$\left[\begin{array}{ll} 5.8570\\ -2.9261\\ 2.3857\\2.4132\end{array}\right]$
 &$\left[\begin{array}{ll}  0.4148\\    0.5852\end{array}\right]$
 &$\left[\begin{array}{ll} 0.1720 &   0.2427\\
    0.2427 &   0.3425\end{array}\right]$\\
\hline
 $( 0.0430\    0.1690\    0.6491\    0.1389)$
 &$\left[\begin{array}{ll}  7.0000\\  -5.0000\\1.0000\\3.0000\end{array}\right]$
 &$\left[\begin{array}{ll} 0.0000\\    1.0000\end{array}\right]$
 &$\left[\begin{array}{ll} 0.0000   & 0.0000\\
   0.0000  &  1.0000\end{array}\right]$\\
\hline
 $(0.2259\    0.1707\    0.2277\    0.3757)$
 &$\left[\begin{array}{ll}  6.0000\\  0.0000\\2.0000\\1.0000\end{array}\right]$
 &$\left[\begin{array}{ll} 1.0000\\    0.0000\end{array}\right]$
 &$\left[\begin{array}{ll} 1.0000   & 0.0000\\
   0.0000  &  0.0000\end{array}\right]$\\
\hline
 $(0.2920\    0.4317\    0.0155\    0.2609)$
 &$\left[\begin{array}{ll}  5.7480\\ -2.4901\\ 2.5020\\ 2.2460\end{array}\right]$
 &$\left[\begin{array}{ll} 0.5020\\    0.4980\end{array}\right]$
 &$\left[\begin{array}{ll}  0.2520  &  0.2500\\
    0.2500  &  0.2480\end{array}\right]$\\
\hline
 $(0.2607\    0.5944\    0.0225\    0.1224)$
 &$\left[\begin{array}{ll}  5.7480\\ -2.4901\\ 2.5020\\ 2.2460\end{array}\right]$
 &$\left[\begin{array}{ll} 0.5020\\    0.4980\end{array}\right]$
 &$\left[\begin{array}{ll}  0.2520  &  0.2500\\
    0.2500  &  0.2480\end{array}\right]$\\
\hline\hline
\end{tabular}
\label{tab: Optimal numerical results for constructed problem in
(DNNP)}
\end{table}

It is very easy to verify that the given four objective functions
are all nonconvex by using MATLAB function \texttt{eig}.
Furthermore, we obtain problem (SOQP) is nonconvex with
corresponding coefficients $\lambda$ which proposed in Table
\ref{tab: Optimal numerical results for constructed problem in
(DNNP)} by using \texttt{eig}. Hence, we use our method to solve
this problem. The corresponding optimal numerical results are
reported in Table \ref{tab: Optimal numerical results for
constructed problem in (DNNP)}. In Table \ref{tab: Optimal numerical
results for constructed problem in (DNNP)}, Example \ref{constructed
problem} is transformed into problem (DNNP), and then is solved with
nine different weighted coefficient $\lambda$. The results of $x^*$
and $X^*$ in Table \ref{tab: Optimal numerical results for
constructed problem in (DNNP)} show that $X^*=x^*(x^*)^T$ holds for
nine different weighted coefficients. Thus, from Theorem \ref{theo:
optimal solution between problems} and Remark \ref{rem: optimal
solution between problems}, it holds that each optimal solution
$x^*$ in Table \ref{tab: Optimal numerical results for constructed
problem in (DNNP)} also is Pareto optimal solution for Example
\ref{constructed problem}. Moreover, we obtain weakly Pareto optimal
solutions for Example \ref{constructed problem} when the weighted
coefficient $\lambda$ is chosen appropriately. For instance, if
$\lambda$ is chosen as $(0.3500, 0.1966, 0.2511, 0.2023)$ or $(
0.0430, 0.1690, 0.6491, 0.1389)$, then weakly Pareto optimal
solution $(0.0000, 1.0000)^T$ is obtained by using our method.

\begin{example}\label{random problem} In this test problem, we set $n=5$, $m=4$ and
$p=5$. The corresponding coefficients $Q_i,\ c_i,\ i=1,\ldots,5$ and
$A,\ b$ are generated by the functions
\texttt{tril(randn(n,n),-1)+triu(randn(n,n)',0)},
\texttt{randn(n,1)}, \texttt{randn(n,m)} and \\
\texttt{randn(m,1)},
respectively, and the details can be seen in Table \ref{tab:
Coefficient data for random problem}.
\end{example}
\begin{table}[htbp]
\centering \caption{Coefficients for Example \ref{random
problem}}\scriptsize
\begin{tabular}{ lllllllllllllllllllllllllll }
 \hline\hline
   $Q_1=\left[\begin{array}{lllllll}0.4056   &-1.4193   &-0.7294    &1.1473
&0.5979\\
-1.4193   &-2.2033   &-0.5712    &0.2140    &0.9424\\
-0.7294   &-0.5712    &0.3062   &-1.1723   &-0.9610\\
1.1473    &0.2140   &-1.1723   &-0.9000   &-0.2857\\
0.5979    &0.9424   &-0.9610   &-0.2857 &0.6103\end{array}\right]$
&$c_1=\left[\begin{array}{ll}0.7907\\  0.2877\\  0.0032\\  0.3656\\  3.5267\end{array}\right]$\\
 \hline
  $Q_2=\left[\begin{array}{lllll}0.0591   &-1.4669   &-1.6258   &-1.9648    &2.6052\\
    -1.4669    &0.2570   &-0.9742   &-1.1464    &0.5476\\
    -1.6258   &-0.9742   &-0.4494   &-0.0843   &-1.9920\\
-1.9648   &-1.1464   &-0.0843   &-0.3909    &0.4092\\
 2.6052 &0.5476
&-1.9920    &0.4092    &1.3018\end{array}\right]$&$c_2=\left[\begin{array}{ll}-0.1124\\  -1.5566\\  1.9151\\  0.6098\\  -0.6479\end{array}\right]$\\
 \hline
  $Q_3=\left[\begin{array}{lllll}-0.5936    &0.4364   &-0.5044    &0.1021
&1.1963\\
0.4364   &-1.0368   &-0.8571   &-0.1699   &-0.1917\\
-0.5044   &-0.8571    &1.2665   &-0.2512   &-0.2046\\
     0.1021   &-0.1699   &-0.2512   &-0.3862    &0.5256\\
1.1963   &-0.1917   &-0.2046    &0.5256 &0.8175\end{array}\right]$
&$c_3=\left[\begin{array}{ll}2.6173\\  0.5510\\  0.2942\\  -0.7778\\  -1.0649\end{array}\right]$\\
\hline $Q_4=\left[\begin{array}{llllll}0.4902    &0.7653 &0.7783
&-1.4803
&0.5404\\
 0.7653   &-0.7603   &-0.6936    &1.2815   &-0.8097\\
  0.7783   &-0.6936    &2.0108    &0.0256    &0.3083\\
-1.4803    &1.2815    &0.0256    &0.5301   &-0.9521\\
 0.5404   &-0.8097    &0.3083   &-0.9521
 &-0.4506\end{array}\right]$
 &$c_4=\left[\begin{array}{ll}-1.7684\\  -0.4229\\  -1.0531\\  0.6478\\  -0.3176\end{array}\right]$\\
 \hline
$Q_5=\left[\begin{array}{lllllll}0.1092   &-0.2506   &-0.1899   &-1.0329   &-0.3233\\
    -0.2506    &1.7447   &-1.1605    &2.3774    &1.5261\\
    -0.1899   &-1.1605   &-0.6987    &0.8328   &-0.6946\\
    -1.0329    &2.3774    &0.8328    &0.8967    &0.5047\\
    -0.3233    &1.5261   &-0.6946    &0.5047
    &1.1867\end{array}\right]$
    &$c_5=\left[\begin{array}{ll}1.7690\\  1.5106\\  0.1640\\  -0.2828\\  1.1522\end{array}\right]$\\
\hline $A=\left[\begin{array}{llllll} -1.1465   &-0.6718 &0.5530
&0.2695
&1.0393\\
    0.6737    &0.5756   &-0.4234   &-2.5644    &0.9109\\
   -0.6691   &-0.7781    &0.3616    &0.4659   &-0.2397\\
   -0.4003   &-1.0636   &-0.3519    &1.8536
   &0.1810\end{array}\right]$
   &$b=\left[\begin{array}{ll}0.2442 \\ 0.0964 \\ -0.8305 \\ -0.3523\end{array}\right]$\\
 \hline\hline
\end{tabular}
\label{tab: Coefficient data for random problem}
\end{table}
In order to verify that whether the given objective functions in
Table \ref{tab: Coefficient data for random problem} are nonconvex
functions or not, the corresponding eigenvalues for $Q_i,\
i=1,\ldots,5$ are given in Table \ref{tab: Eigenvalues}. The results
in Table \ref{tab: Eigenvalues} show that five objective functions
are all nonconvex. Thus, we will use the proposed method to solve
this problem. The results are given in Tables \ref{tab: Optimal
solutions of single objective problem
 for random problem} and
\ref{tab: Numerical results for random problem}.
\begin{table}[htbp]
\centering \caption{Eigenvalues for $Q_i,\ i=1,\ldots,5$ in Example
\ref{random problem}}\scriptsize
\begin{tabular}{ lllllllllllllllllllllllllll }
 \hline\hline
 Quadratic matrices    &Eigenvalues\\
 \hline
 $Q_1$       &$(-3.3892,   -1.8503,   -0.0780,    1.1180,    2.4184)$\\
 \hline
 $Q_2$       &$( -4.2248,   -1.8386,    0.7234,    1.3339,    4.7837)$\\
 \hline
 $Q_3$       &$( -1.6777,   -1.1336,   -0.3917,    1.1973,    2.0731)$\\
 \hline
 $Q_4$       &$(-2.5335,   -1.0738,    0.6848,    1.6750,    3.0677)$\\
 \hline
 $Q_5$      &$( -2.3512,   -0.6382,    0.1606,    1.3837,    4.6838)$\\
\hline\hline
\end{tabular}
\label{tab: Eigenvalues}
\end{table}
\begin{table}[htbp]
\centering \caption{Optimal numerical results for each objective
function of Example \ref{random problem}}\scriptsize
\begin{tabular}{ lllllllllllllllllllllllllll }
 \hline\hline
              &&&&&FV   &&&&&$x^*$ \\
 \hline
 $F_1(x^*)$      &&&&&$2.1900$      &&&&&$(0,\     1.2165,\     0.2253,\     0.4747,\     0.7783)^T$ \\
 \hline
 $F_2(x^*)$      &&&&&$-9.2542$
 &&&&&$(2.7216,\         0,\    2.7773,\    0.7727,\    1.5592)^T$     \\
 \hline
 $F_3(x^*)$      &&&&&$-1.3755$
 &&&&&$( 0,\    1.2165,\    0.2253,\    0.4747,\    0.7783)^T$     \\
 \hline
 $F_4(x^*)$      &&&&&$-2.2276$
 &&&&&$(0.4915,\    0.9969,\    0.6861,\    0.5285,\    0.9193)^T$     \\
 \hline
 $F_5(x^*)$      &&&&&$0.6759$
 &&&&&$( 2.7216,\         0,,\    2.7773,\    0.7727,\    1.5592)^T$     \\
\hline\hline
\end{tabular}
\label{tab: Optimal solutions of single objective problem
 for random problem}
\end{table}

Table \ref{tab: Optimal solutions of single objective problem
 for random problem} shows that the optimal numerical results for
 each objective function of Example \ref{random problem}.
 The results in Table \ref{tab: Optimal solutions of single objective problem
 for random problem} show that we can not find a single solution
that simultaneously optimizes these five objective functions. Note
that these five optimal solutions are also weakly Pareto optimal
solutions for Example \ref{random problem}.

\begin{table}[htbp]
\centering \caption{Optimal numerical results for Example
\ref{random problem} in (DNNP)}\scriptsize
\begin{tabular}{ lllllllllllllllllllllllllll }
 \hline\hline
   $\lambda$        &FV            &$x^*$                 &$X^*$\\
 \hline
   $\left[\begin{array}{ll}0.3317\\
    0.1522\\
    0.3480\\
    0.1217\\
    0.0464\end{array}\right]$
   &$\left[\begin{array}{llllll}
                     2.1900\\
                    -1.7656\\
                    -1.3755\\
                     -1.7914\\
                     7.0248\end{array}\right]$
  &$\left[\begin{array}{ll}
                                       0.0000\\1.2165\\
                                       0.2253\\0.4747\\
                                       0.7783\end{array}\right]$
  &$\left[\begin{array}{lllll}
               0.0000& 0.0000& 0.0000& 0.0000& 0.0000\\
               0.0000& 1.4799& 0.2741& 0.5775& 0.9469\\
               0.0000& 0.2741& 0.0508& 0.1070& 0.1753\\
               0.0000& 0.5775& 0.1070& 0.2254& 0.3695\\
               0.0000& 0.9469& 0.1753& 0.3695& 0.6058\end{array}\right]$\\
 \hline
  $\left[\begin{array}{ll}
   0.0712\\
    0.1820\\
    0.0930\\
    0.4635\\
    0.1903\end{array}\right]$
  &$\left[\begin{array}{lll}
                    2.2944\\
                      -2.3834\\
                      0.1927\\
                      -2.2196\\
                       6.2356\end{array}
                      \right]$
  &$\left[\begin{array}{ll} 0.4250\\
    1.0266\\
    0.6238\\
    0.5213\\
    0.9003\end{array}\right]$
  &$\left[\begin{array}{lllllll}
               0.1806 &   0.4363  &  0.2651 &   0.2215  &  0.3826\\
    0.4363 &   1.0538 &   0.6403  &  0.5351  &  0.9242\\
    0.2651 &   0.6403 &   0.3891  &  0.3252  &  0.5616\\
    0.2215 &   0.5351 &   0.3252  &  0.2717  &  0.4693\\
    0.3826 &   0.9242 &   0.5616  &  0.4693  &  0.8105\end{array}\right]$\\
 \hline
 $\left[\begin{array}{ll}
 0.0290\\
    0.6371\\
    0.0595\\
    0.1692\\
    0.1052\end{array}\right]$
 &$\left[\begin{array}{ll}
                    2.3045\\
                    -2.4426\\
                    0.3234\\
                    -2.2257\\
                    6.1681\end{array}
                    \right]$
 &$\left[\begin{array}{ll}0.4595\\
    1.0111\\
    0.6562\\
    0.5251\\
    0.9102\end{array}\right]$
 &$\left[\begin{array}{llllll}
               0.2111  &  0.4646  &  0.3015  &  0.2413  &  0.4182\\
    0.4646  &  1.0224  &  0.6635  &  0.5309  &  0.9203\\
    0.3015  &  0.6635  &  0.4305  &  0.3445  &  0.5972\\
    0.2413  &  0.5309  &  0.3445  &  0.2757  &  0.4779\\
    0.4182  &  0.9203   & 0.5972   & 0.4779   & 0.8284\end{array}\right]$\\
 \hline
  $\left[\begin{array}{ll}
  0.0472\\
    0.0559\\
    0.0715\\
    0.0157\\
    0.8098\end{array}\right]$
  &$\left[\begin{array}{ll}
                     3.5136\\
                     -9.2542\\
                     9.9301\\
                      6.7545\\
                     0.6759\end{array}\right]$
  &$\left[\begin{array}{ll}2.7216\\
    0.0000\\
    2.7773\\
    0.7727\\
    1.5592\end{array}\right]$
  &$\left[\begin{array}{llllll}
               7.4071 &   0.0000 &   7.5586  &  2.1030 &   4.2434\\
    0.0000   & 0.0000  &  0.0000 &   0.0000  &  0.0000\\
    7.5586   & 0.0000  &  7.7132  &  2.1460  &  4.3302\\
    2.1030   & 0.0000  &  2.1460  &  0.5971  &  1.2048\\
    4.2434   & 0.0000  &  4.3302  &  1.2048  &  2.4310\end{array}\right]$\\
 \hline
  $\left[\begin{array}{ll}
   0.1911\\
    0.0504\\
    0.5484\\
    0.0095\\
    0.2006\end{array}\right]$
  &$\left[\begin{array}{ll}
                    2.1900\\
                    -1.7656\\
                     -1.3755\\
                     -1.7914\\
                   7.0248\end{array}\right]$
  &$\left[\begin{array}{ll}0.0000\\
    1.2165\\
    0.2253\\
    0.4747\\
    0.7783\end{array}\right]$
  &$\left[\begin{array}{lllll}
               0.0000 &   0.0000  &  0.0000  &  0.0000 &   0.0000\\
    0.0000 &   1.4799  &  0.2741 &   0.5775  &  0.9469\\
    0.0000 &   0.2741 &   0.0508  &  0.1070 &   0.1753\\
    0.0000  &  0.5775  &  0.1070  &  0.2254  &  0.3695\\
    0.0000  &  0.9469  &  0.1753 &   0.3695  &  0.6058\end{array}\right]$\\
 \hline
  $\left[\begin{array}{ll}
   0.2166\\
    0.6291\\
    0.0149\\
    0.0433\\
    0.0961\end{array}\right]$
  &$\left[\begin{array}{ll}
                       3.5136\\
                      -9.2542\\
                       9.9301\\
                       6.7545\\
                       0.6759\end{array}\right]$
  &$\left[\begin{array}{ll}2.7216\\
    0.0000\\
    2.7773\\
    0.7727\\
    1.5592\end{array}\right]$
  &$\left[\begin{array}{lllll}
               7.4071  &  0.0000 &   7.5586  &  2.1030  &  4.2434\\
    0.0000 &   0.0000  &  0.0000 &   0.0000  &  0.0000\\
    7.5586 &   0.0000  &  7.7132  &  2.1460  &  4.3302\\
    2.1030  &  0.0000  &  2.1460  &  0.5971 &   1.2048\\
    4.2434  &  0.0000  &  4.3302  &  1.2048  &  2.4310\end{array}\right]$\\
  \hline
  $\left[\begin{array}{ll}
   0.2583\\
    0.0041\\
    0.0905\\
    0.5958\\
    0.0513
\end{array}\right]$
  &$\left[\begin{array}{ll}
                       2.2080\\
                      -1.8738\\
                        -1.0651\\
                       -1.9303\\
                       6.8716\end{array}\right]$
  &$\left[\begin{array}{ll}0.0857\\
    1.1782\\
    0.3057\\
    0.4841\\
    0.8029\end{array}\right]$
  &$\left[\begin{array}{lllll}
               0.0073 &   0.1010 &   0.0262 &   0.0415 &   0.0688\\
    0.1010  &  1.3882  &  0.3601 &   0.5704 &   0.9460\\
    0.0262  &  0.3601  &  0.0934 &   0.1480 &   0.2454\\
    0.0415  &  0.5704 &   0.1480 &   0.2344 &   0.3887\\
    0.0688  &  0.9460  &  0.2454  &  0.3887 &   0.6447
\end{array}\right]$\\
\hline\hline
\end{tabular}
\label{tab: Numerical results for random problem}
\end{table}

Note that we can verify that problem (SOQP) is nonconvex with seven
different choices of weighted coefficient $\lambda$ which show in
Table \ref{tab: Numerical results for random problem}. So, Example
\ref{random problem} can be solved by using the proposed method. The
results of $x^*$ and $X^*$ in Table \ref{tab: Numerical results for
random problem} show that $X^*=x^*(x^*)^T$ holds for these seven
different cases of weighted coefficient $\lambda$. Hence, we can
conclude that each $x^*$ in Table \ref{tab: Numerical results for
random problem} also is Pareto optimal solution for Example
\ref{random problem}. Furthermore, compare the results of FV and
$x^*$ in Tables \ref{tab: Optimal solutions of single objective
problem
 for random problem} and \ref{tab: Numerical results for random
 problem}, note that some weakly Pareto optimal solutions of Example \ref{random problem} can be
 obtained by using our method. For example, if $\lambda=(0.3317, 0.1522, 0.3480, 0.1217,
 0.0464)$, then we obtain weakly Pareto optimal solution $(0.0000,
 1.2165, 0.2253, 0.4747, 0.7783)^T$ for Example \ref{random problem}.
\begin{remark}\label{rem: unmerical results}
The results for Example \ref{constructed problem} and Example
\ref{random problem} imply that we not only
 obtain Pareto optimal solutions, but
 also obtain some weakly Pareto optimal solutions for original problem with appropriate choices of weighted coefficient $\lambda$.
 Summarizing these results, we can conclude that our method is
 effective for solving some problems (MOQP).
\end{remark}
\begin{example}\label{Portfolio Selection Problem} (Portfolio Selection Problem) This problem is taken from
\cite{xl2002}. It is a practical portfolio selection problem in
which objective function $F(x)$ has the following expression
$$F(x)=(\frac{1}{2}x^TQ_1x,\ -c_2^Tx,\ -c_3^Tx)^T,$$
where symmetric matrix $Q_1$ is called the risk matrix, $c_2$
denotes the return rate vector, $c_3$ is a given weighting vector
and its element is a function of corresponding security liquidity.
The corresponding coefficients are given in Table \ref{tab: data for
Portfolio Selection Problem }.
\begin{table}[htbp]
\centering \caption{Coefficients for Example \ref{Portfolio
Selection Problem}}\scriptsize
\begin{tabular}{ lllllllllllllllllllllllllll }
 \hline\hline
$Q_1=\left[\begin{array}{llllllllllllllllllllll}
12.35  &-8.64  &9.88  &-8.43  &9.42  &-1.47  &4.29  &-0.96  &15.2  &0\\
 -8.64 &6.42  &-8.15  &5.13   &4.56  &5.83  &-6.72  &0.78   &3.42  &0\\
9.88  &-8.15  &18.27  &-8.14  &2.32  &-4.12 &-1.85  &1.34  &0.795  &0\\
-8.43 &5.13  &-8.14  &12.6   &-8.43  &7.47  &-4.67  &-0.679 &8.09  &0\\
9.42  &4.56  &2.32  &-8.43  &17.07  &7.33  &-3.18  &-0.547  &0.323 &0\\
-1.47  &5.83  &-4.12 &7.47  &7.33  &13.38  &-3.84  &-0.478  &0.616 &0\\
4.29  &-6.27  &-1.85  &-4.67  &-3.18  &-3.84  &3.41  &0.107  &0.464 &0\\
-0.96  &-0.78  &1.34  &-0.679  &-0.547  &-0.478  &0.107  &1.33 &11.23 &0\\
15.2  &3.42  &0.795  &0.809  &0.323  &0.616  &0.464  &11.23  &2.45  &0\\
0  &0  &0  &0  &0  &0  &0  &0  &0  &0 \end{array}\right]$\\
\hline $c_2=[11.11\ \  17.78\ \ 2.22\ \ 18.28\ \ 12.6\ \ 10.15\ \
6.72\ \ 5.75\ \ 6.1\ \ 2.25]^T$\\
\hline $c_3=[0.069\ \ 0.092\ \ 0.046\ \ 0.096\ \ 0.086\ \ 0.068\ \
0.06\ \ 0.145\ \
0.151\ \  0.197]^T$\\
\hline
$A=[1\ 1\ 1\ 1\ 1\ 1\ 1\ 1\ 1\ 1]\ \ \ b=1$\\
 \hline\hline
\end{tabular}
\label{tab: data for Portfolio Selection Problem }
\end{table}
\end{example}
Note that the objective functions contain only one quadratic
function $\frac{1}{2}x^TQ_1x$, by using \texttt{eig} function of
MATLAB, it is easy to verify that function $\frac{1}{2}x^TQ_1x$ is
nonconvex. Thus, we use the proposed method to solve this problem.
The corresponding optimal results are given in Table \ref{tab:
Optimal values for portfolio selection problem I}.
\begin{table}[htbp]
\centering \caption{Optimal numerical results for Example
\ref{Portfolio Selection Problem} }\scriptsize
\begin{tabular}{ lllllllllllllllllllllllllll }
 \hline\hline
$\lambda$                   &FV  &$x^*$  \\
 \hline
   $\left[\begin{array}{ll}
   0.3482\\
    0.1655\\
    0.4863\end{array}\right]$
   &$\left[\begin{array}{ll}
    0.1894\\
    -15.6800\\
    -0.0850\end{array}\right]$
   &$(0.3240,
    0.5567,
    0.0000,
    0.1194,
    0.0000,
    0.0000,
    0.0000,
    0.0000,
    0.0000,
    0.0000)^T$\\
 \hline
 $\left[\begin{array}{ll}
  0.5141\\
    0.1931\\
    0.2927
\end{array}\right]$
 &$\left[\begin{array}{ll}
 0.1320\\
 -15.5462\\
   -0.0846\end{array}\right]$
 &$(0.3433,
    0.5466,
    0.0000,
    0.1101,
    0.0000,
    0.0000,
    0.0000,
    0.0000,
    0.0000,
    0.0000)^T$\\
 \hline
 $\left[\begin{array}{ll}
 0.6167\\
    0.3087\\
    0.0746\end{array}\right]$
 &$\left[\begin{array}{ll}
  0.2026\\
  -15.7070\\
   -0.0851\end{array}\right]$
 &$(0.3200,
    0.5590,
    0.0000,
    0.1210,
    0.0000,
    0.0000,
    0.0000,
    0.0000,
    0.0000,
    0.0000)^T$\\
\hline\hline
\end{tabular}
\label{tab: Optimal values for portfolio selection problem I}
\end{table}

In Table \ref{tab: Optimal values for portfolio selection problem
I}, Example \ref{Portfolio Selection Problem} is solved with three
different values of weight $\lambda$. The first element and the
absolute value of the second element of FV in Table \ref{tab:
Optimal values for portfolio selection problem I} denote the
expectation risk and return, respectively. The results in Table
\ref{tab: Optimal values for portfolio selection problem I} show
that we obtain the lower risk and higher return when weight
$\lambda$ is chosen appropriately. Moreover, we also compare with
the results in \cite{xl2002}, which are show in Table \ref{tab:
Optimal results for portfolio selection problem I in xu_li}.
\begin{table}[htbp]
\centering \caption{Optimal numerical results for Example
\ref{Portfolio Selection Problem} in \cite{xl2002} }\scriptsize
\begin{tabular}{ lllllllllllllllllllllllllll }
 \hline\hline
                 FV  &$x^*$  \\
 \hline
   $\left[\begin{array}{ll}
   0.6823\\
    -7.5079\\
    -0.1343\end{array}\right]$
   &$(0.0453,
    0.0706,
    0.0273,
    0.0703,
    0.078,
    0.045,
    0.0534,
    0.169,
    0.1665,
    0.2746)^T$\\
\hline\hline
\end{tabular}
\label{tab: Optimal results for portfolio selection problem I in
xu_li}
\end{table}

The results of FV in Tables \ref{tab: Optimal values for portfolio
selection problem I} further imply that the risk and return are more
comparable with the results in Table \ref{tab: Optimal results for
portfolio selection problem I in xu_li}. For example, when
$\lambda=(0.3482,    0.1655, 0.4863)$, we obtain lower risk $0.1894$
and higher return $15.6800$, which are more comparable with the
results of \cite{xl2002} $0.6823$ and $7.5079$, respectively.
Furthermore, note that the optimal solutions $x^*$ obtained by using
our method are sparse, which imply that we mainly focus on some
kinds of important stocks. Hence, we can put together the limited
money, and invest these money in some of important stocks to obtain
more satisfied return with lower risk.

We also notice that there are infinite choices of weight $\lambda$,
and not all choices of weight $\lambda$ are reasonable for each
investor. How to select weight $\lambda$ to lower risk and higher
return depends on the investors's preference. Therefore, it is
reasonable that investors participate in decision making and
continuously revise their preferences according to practical
conditions. This also shows that our method is promising in solving
portfolio selection problems.

\section{Concluding remarks}\label{sec: Concluding remarks}
In this paper, a class of (MOQP) problems is discussed. By using the
linear weighted sum method to deal with quadratic objective
functions, problem (MOQP) is transformed into problem (SOQP), which
is nonconvex in general. Then, taking advantage of lifting
techniques, problem (SOQP) is equivalently reformulated as problem
(CP) which is a convex programming problem but NP-hard in general. A
computable relaxed convex problem (DNNP) for problem (CP) is
obtained by using doubly nonnegative relaxation method. Moreover,
based on the characteristics of optimal solutions for problem
(DNNP), a sufficient condition for (weakly) Pareto optimal solutions
for problem (MOQP) is proposed. Finally, the numerical results of
two problems and a practical portfolio selection problem show that
the proposed method is effective and promising.


\medskip
Received xxxx 20xx; revised xxxx 20xx.
\medskip

\end{document}